\documentclass[a4paper,twoside,reqno]{amsart}
\usepackage{amssymb}
\usepackage[mathscr]{eucal}

\usepackage{hyperref}

\usepackage[retainorgcmds]{IEEEtrantools}

\usepackage{tikz}
\usetikzlibrary{decorations.pathreplacing}

\usepackage{stmaryrd}

\theoremstyle{plain}
\newtheorem{prop}{Proposition}[section]
\newtheorem{thm}[prop]{Theorem}
\newtheorem{cor}[prop]{Corollary}
\newtheorem{lem}[prop]{Lemma}

\theoremstyle{definition}
\newtheorem{dfn}[prop]{Definition}
\newtheorem{rem}[prop]{Remark}
\newtheorem{lab}[prop]{}
\newtheorem{ex}[prop]{Example}

\newcommand{\C}{{\mathbb{C}}}
\renewcommand{\P}{{\mathbb{P}}}
\newcommand{\R}{{\mathbb{R}}}
\newcommand{\N}{{\mathbb{N}}}
\newcommand{\Z}{{\mathbb{Z}}}

\newcommand{\scrF}{{\mathscr{F}}}

\newcommand{\scrU}{{\mathscr{U}}}

\newcommand{\x}{\texttt{x}}

\newcommand{\sfS}{\mathsf{S}}
\newcommand{\sfH}{\mathsf{H}}
\renewcommand{\H}{\mathcal{H}}

\DeclareMathOperator{\im}{im}
\DeclareMathOperator{\interior}{int}
\DeclareMathOperator{\relint}{relint}
\DeclareMathOperator{\rk}{rk}

\DeclareMathOperator{\spn}{span}
\DeclareMathOperator{\conv}{conv}
\DeclareMathOperator{\suppface}{suppface}

\DeclareMathOperator{\id}{id}
\DeclareMathOperator{\Newt}{Newt}
\DeclareMathOperator{\Hom}{Hom}
\DeclareMathOperator{\Gram}{Gram}
\DeclareMathOperator{\Ex}{Ex}

\renewcommand{\subset}{\subseteq}

\newcommand{\du}{{\scriptscriptstyle\vee}}
\newcommand{\ol}[1]{\overline{#1}}
\newcommand{\plus}{{\scriptscriptstyle+}}

\renewcommand{\setminus}{\smallsetminus}
\renewcommand{\epsilon}{\varepsilon}
\renewcommand{\theta}{\vartheta}

\renewcommand{\choose}[2]{\genfrac(){0pt}{}{#1}{#2}}

\begin{document}

\title
[Polyhedral faces in Gram spectrahedra of binary forms] 
{Polyhedral faces in Gram spectrahedra \\ of binary forms}

\author
 {Thorsten Mayer}
\address
 {Fachbereich Mathematik und Statistik, Universit\"at Konstanz,
 78457 Konstanz, Germany}
\email
 {thorsten.mayer@uni-konstanz.de}

\begin{abstract}
We analyze both the facial structure of the Gram spectrahedron $\Gram(f)$ and of the Hermitian Gram spectrahedron $\H^\plus(f)$ of a nonnegative binary form $f \in \R[x, y]_{2d}$.  
We show that if $F \subset \H^\plus(f)$ is a polyhedral face of dimension $k$ then $\choose{k+1}{2} \leq d$. 
Conversely, for all $k \in \N$ and $d \geq \choose{k+1}{2}$ we show that the Hermitian Gram spectrahedron of a general positive binary form $f \in \R[x, y]_{2d}$ with distinct roots contains a face $F$ which is a $k$-simplex and whose extreme points are rank-one tensors. 
For all $k \in \N$ and $d \geq (k+1)^2$ the (symmetric) Gram spectrahedron of a general positive binary form $f \in \R[x, y]_{2d}$ contains a polyhedral face $F$ with $(\rk(F), \dim(F)) = (2(k+1), k)$.  
\end{abstract}

\thanks
{I was supported by a doctoral scholarship of Studienstiftung des deutschen Volkes. \\
This work is part of my ongoing doctoral research project.
I wish to thank my doctoral supervisor Claus Scheiderer for his continuous help and support.
I would also like to thank Christoph Schulze and Julian Vill for helpful discussions on the subject.}

\maketitle

%-------------------------------------------------------------------%

\section*{Introduction}
A form $f \in \R[\x] = \R[x_0, \dots, x_n]$ is a \emph{sum of squares} if $f = p_1^2 + \dots + p_r^2$ for some $p_1, \dots, p_r \in \R[\x]$. 
If $f$ is known to be a sum of squares, there are usually many inequivalent ways of representing $f$ as such. 
The Gram spectrahedron $\Gram(f)$ parametrizes the sum-of-squares representations of $f$, modulo orthogonal equivalence. 
We can understand this spectrahedron as the intersection of the cone of positive semidefinite (\emph{psd}) matrices with an affine-linear space as Chua, Plaumann, Sinn and Vinzant \cite{CPSV} do in their survey on Gram spectrahedra. 
Or, equivalently, we can use a coordinate-free approach introduced by Scheiderer in \cite{Sch} and consider Gram spectrahedra as sets of psd symmetric tensors.    
We will recall his methods later on.

In this paper we focus on the facial structure of Gram spectrahedra of binary forms.
We examine the relationship between rank and dimension of faces $F \subset \Gram(f)$ and we show which pairs $(\rk(F), \dim(F))$ can occur. 
In particular, we will observe large dimension gaps.

Another possibility to certify nonnegativity of a form $f \in \R[\x]$ is to write it as a \emph{Hermitian sum of squares},  
i.e.~$f = p_1 \ol{p_1} + \dots + p_r \ol{p_r}$, where $p_1, \dots, p_r \in \C[\x]$. 
The Hermitian Gram spectrahedron $\H^\plus(f)$, which parametrizes the representations of $f$ as a Hermitian sum of squares, is an interesting object of convex algebraic geometry on its own terms and is also discussed in \cite{CPSV}. 
But it turns out that we can also use it to gain a better understanding for the symmetric Gram spectrahedron.

As Laurent and Poljak write in \cite{LP}, a polyhedral face is, in some sense, the most 'nonsmooth part' of the boundary of a spectrahedron.
In addition, polyhedra are the simplest examples of spectrahedra.  
For these reasons, we are interested in polyhedral faces of Gram spectrahedra. 
We give bounds on the dimension of such faces.  
Conversely, we show that the Hermitian Gram spectrahedron of a general positive binary form contains a simplex face of the largest possible dimension (Theorem \ref{thm:generic_herm_case}).
In \autoref{sec:polyhedral_faces_real}, we use our findings from the Hermitian case to show an analogous result in the real symmetric case.

%-------------------------------------------------------------------%

\section{Dimensions of faces in the symmetric Gram spectrahedron}
\label{sec:real_Gram_spectrahedra}

In this section we determine possible dimensions of a face of $\Gram(f)$ as a function of its rank. 
To begin with, we briefly recall the coordinate-free approach to Gram spectrahedra introduced by Scheiderer. 
For details and explanations see Section 2 and 3 in \cite{Sch}.

\begin{lab} 
Let $A$ be an $\R$-algebra and let $\sfS_2 A \subset A \otimes A$ denote the space of symmetric tensors, i.e.~tensors that are invariant under the involution $p \otimes q \mapsto q \otimes p$.
The multiplication map $\mu \colon A \times A \to A$, $(p, q)\mapsto pq$ is $\R$-bilinear and induces an $\R$-linear map $\sfS_2 A \to A$.
Let $V \subset A$ be a subspace of finite dimension. 
Given $f \in A$, the symmetric tensors $\theta \in \sfS_2 V$ with $\mu(\theta) = f$ are called the \emph{Gram tensors} of $f$, relative to $V$.
After fixing a linear basis of $V$, the space $\sfS_2 V$ gets identified with the space of symmetric $N \times N$-matrices ($N = \dim(V)$). 
The fact that every real symmetric matrix can be diagonalized implies that every $\theta \in \sfS_2 V$ can be written as $\sum_{i=1}^r \epsilon_i (p_i \otimes p_i)$, with $r \geq 0$, $\epsilon_i = \pm 1$ and $p_1, \dots, p_r \in V$ linearly independent. 
In this case, the \emph{image} of $\theta$ is $\im(\theta) = \spn(p_1, \dots, p_r) \subset V$ and the \emph{rank} of $\theta$ is $\rk(\theta) = \dim \im(\theta) = r$. 
Furthermore, $\theta$ is \emph{positive semidefinite (psd)}, written $\theta \succeq 0$, if and only if all $\epsilon_i = 1$.
We write $\sfS_2^\plus V = \{ \theta \in \sfS_2 V : \theta \succeq 0 \}$ for the cone of psd tensors.  
The \emph{(symmetric) Gram spectrahedron} of $f$, relative to $V$, is the set of all positive semidefinite Gram tensors of $f$ in $\sfS_2 V$, i.e. \[
\Gram_V(f) := \sfS_2^\plus V \cap \mu^{-1}(f).
\]
\end{lab}  

\begin{lab}
Ramana and Goldman \cite{RG} provide useful tools for the analysis of the facial structure of spectrahedra. 
In his coordinate-free review of their results, Scheiderer calls a linear subspace $U$ of $V$ \emph{facial}, or a \emph{face subspace} (for the given spectrahedron $\Gram_V(f)$), if there exists $\theta \in \Gram_V(f)$ with $U = \im(\theta)$. 
By Proposition 2.10 in \cite{Sch}, there is a natural inclusion-preserving bijetion between the nonempty faces $F$ of $\Gram_V(f)$ and the face subspaces $U \subset V$ for $\Gram_V(f)$, given by 
\begin{IEEEeqnarray*}{rCl}
F & \mapsto & \scrU(F) := \sum_{\theta \in F} \im(\theta), \\
\{ \theta \in \Gram_V(f) : \im(\theta) \subset U \} =: \scrF(U) & \mapsfrom & U.
\end{IEEEeqnarray*}
Scheiderer also mentions the following consequences of this result:
Given a face subspace $U \subset V$, the relative interior of $\scrF(U)$ is $\{\theta \in \Gram_V(f) : \im(\theta) = U\}$. 
Conversely, if $F \subset \Gram_V(f)$ is a face, then $\im(\theta) = \scrU(F)$ and $\rk(\theta) = \dim \scrU(F)$ for every $\theta \in \relint(F)$.
This number is called the \emph{rank} of $F$, denoted $\rk(F)$.
The supporting face of $\theta \in \Gram_V(f)$ is $\scrF(\im(\theta))$.
\end{lab}

We recall another previously known result that is also crucial for this paper.

\begin{prop}[see \cite{Sch}, Prop.~3.6] 
For $U \subset V$ a face subspace for $f$, the face $\scrF(U)$ of $\Gram_V(f)$ has dimension \[
\dim \scrF(U) = \frac{r}{2}(r+1) - s
\]
with $r = \dim(U)$ and $s = \dim(UU)$. 
Here, $UU$ denotes the linear subspace of $A$ spanned by the products $pq$ ($p, q \in U$).
\end{prop}

From now on, we will work with binary forms. 
We write $\Sigma_{2d} = \Sigma \R[x, y]_d^2$ for the sums-of-squares cone in $\R[x, y]_{2d}$. 
For $f \in \Sigma_{2d}$ let $\Gram(f)$ be the Gram spectrahedron of $f$, i.e.~$\Gram(f) := \Gram_V(f)$ with $V = \R[x, y]_d$.  
Note that $\dim \Gram(f) = \choose{d}{2}$ if and only if $f \in \interior(\Sigma_{2d})$. 
Otherwise, by Proposition 5.5 in the paper by Choi, Lam and Reznick  \cite{CLR}, $\Gram(f)$ does not contain any positive definite tensor (i.e.~any tensor of rank $\dim(V) = d+1$). \\
Let $f \in \Sigma_{2d}$ and let $F \subset \Gram(f)$ be a face of rank $r$. 
Let $U = \scrU(F) \subset \R[x, y]_d$ denote the corresponding face subspace.
Since $UU \subset \R[x, y]_{2d}$, the dimension of $F = \scrF(U)$ is at least $\choose{r+1}{2} - (2d+1)$. 
Below, we prove an upper bound.  \\
Essentially, via $f(x, y) \mapsto f(x, 1)$ and $g(x) \mapsto y^{d}g(\frac{x}{y})$, a binary form of degree $d$ is nothing but a univariate polynomial of degree at most $d$. 
Therefore, we will occasionally work with univariate polynomials to simplify notation.

\begin{prop} \label{prop:UU_geq_2r-1}
Let $U \subset \R[x]$ be a subspace of dimension $r$. 
Then the dimension of $UU$ is at least $2r-1$.
\end{prop}
\begin{proof}
We can choose a basis $p_1, \dots, p_r$ of $U$ such that $\deg(p_i) < \deg(p_{i+1})$ for all $i = 1, \dots, r-1$. 
Then \[ 
p_1 p_1, p_1 p_2, p_2 p_2, p_2 p _3, p_3 p_3, \dots, p_{r-1} p_r, p_r p_r
\]
are $2r-1$ polynomials in $UU$ of pairwise different degree. 
Therefore, they are linearly independent.
\end{proof}

In other words, if $U \subset \R[x,y]_d$ is a subspace of codimension $k$, then the codimension of $UU$ in $\R[x,y]_{2d}$ is at most $2k$.

\begin{cor} \label{cor:max_dim_of_faces_real}
Let $f \in \R[x,y]_{2d}$ be a nonnegative binary form and let $F \subset \Gram(f)$ be a face of rank $r$.
Then $\dim(F) \leq \choose{r-1}{2}$. 
\end{cor}
\begin{proof}
Consider the subspace $U = \scrU(F)$ of $ \R[x,y]_d$. By Proposition \ref{prop:UU_geq_2r-1},
\[
\dim(F) = \choose{r+1}{2} - \dim(UU) \leq \choose{r+1}{2} - (2r-1) = \choose{r-1}{2}.  \qedhere
\]
\end{proof}

\begin{rem} 
Let $\rk(F) = r$ and write $k = (d+1)-r$, i.e.~$k = \dim \ker(\theta)$ for $\theta \in \relint(F)$. 
Thereby, we can represent both inequalities for the dimension of $F$ in a uniform way:
 \[
\choose{d-k}{2} - 2k \leq \dim(F) \leq \choose{d-k}{2}. 
\]
\end{rem}

\begin{lab}
Let $f \in \Sigma_6$ ($d = 3$). 
By Corollary \ref{cor:max_dim_of_faces_real}, a proper face of $\Gram(f)$ has dimension at most $\choose{d-1}{2} = 1$.
In fact, Scheiderer has shown that $\Gram(f)$ has no faces of dimension $1$ if $f$ is strictly positive (\cite{Sch}, Prop.~5.3). 
In other words, if $\Gram(f)$ has a rank-$d$ face $F$ of dimension $\choose{d-1}{2}$, i.e.~$U = \scrU(F)$ attains the lower bound $\dim(UU) = 2d-1$,
then $f(\xi) = 0$ for some $\xi \in \P^1(\R)$ and therefore $\Gram(f) = F$. 
We will show that this generalizes to arbitrary degree.
The main part of the proof is the following lemma. 
\end{lab}

\begin{lem} \label{lem:max_face_dim_means_root}
Let $d \geq 3$ and let $U \subset \R[x]_{\leq d}$ be a $d$-dimensional subspace.
If $1 \notin U$ and $\dim(UU) = 2d-1$ then the elements of $U$ have a common real root. 
\end{lem}
\begin{proof}
The case $d = 3$ is discussed in the previous remark.
We assume $d \geq 4$ and proceed by induction on $d$.
We can choose a basis $(p_1, \dots, p_d)$ of $U$ where $p_i = x^i + \lambda_i$ for some $\lambda_i \in \R$.
Consider $U' := \spn(p_1, \dots, p_{d-1}) \subset \R[x]_{\leq d-1}$. 
Since $\deg(p_d) = d$ is bigger than the degree of any element in $U'$, the proof of Proposition \ref{prop:UU_geq_2r-1} shows that $\dim(U'U') \leq \dim(UU)-2 = 2(d-1)-1$.
The right hand side coincides with the lower bound from \ref{prop:UU_geq_2r-1}, so we must have equality.
By induction, $p_1, \dots, p_{d-1}$ have a common root $a \in \R$.
The choice of a simple basis for $U$ makes it easy to read off bases for $U'U'$ and $UU$:
\[
p_1 p_1, p_1 p_2, p_2 p_2, \dots, p_{d-2} p_{d-1}, p_{d-1} p_{d-1} \]
are $2(d-1)-1 = \dim(U'U')$ elements of $U'U'$ of pairwise different degrees, so they constitute a basis for $U'U'$ which is completed to a basis of $UU$ by $p_{d-1} p_d$ and $p_d p_d$.
We consider the representation of $p_1 p_d \in UU$ with respect to this basis. 
Since $d \geq 3$ and for reasons of degree, the coefficients of $p_{d-1} p_d$ resp.~$p_d p_d$ are zero. 
Hence $p_1 p_d \in U'U'$. 
But all elements of $U'U'$ have a double root in $a$, so $(x-a)^2$ divides $p_1 p_d = (x-a) p_d$. 
Therefore, also $p_d(a) = 0$.
\end{proof}

\begin{prop} \label{prop:max_face_dim_means_root}
Let $d \geq 3$ and let $U \subset \R[x,y]_d$ be a $d$-dimensional subspace.
If $\dim(UU) = 2d-1$ then there exists a $\xi \in \P^1(\R)$ such that $p(\xi) = 0$ for all $p \in U$. 
\end{prop}
\begin{proof}
We tackle the problem in its univariate formulation. 
For this purpose, we write $U = \spn(p_1, \dots, p_d) \subset \R[x]_{\leq d}$, $\dim(U) = d$, and assume that $\dim(UU) = 2d-1$.
If $\deg(p_i) < d$ for all $i = 1, \dots, d$, then $(1:0) \in \P^1(\R)$ is a common root of the corresponding degree-$d$ forms $y^d \cdot p_i(\frac{x}{y})$. 
Therefore, we may assume that $p_d$ is monic, $\deg(p_d) = d$ and $\deg(p_i) < d$ for all $i < d$.
Consider $U' := \spn(p_1, \dots, p_{d-1}) \subset \R[x]_{\leq d-1}$. 
The same argument as in the proof of \ref{lem:max_face_dim_means_root} shows that $\dim(U'U') = 2(d-1)-1$. 
By induction, either $p_1, \dots, p_{d-1}$ are of degree smaller than $d-1$ or they have a common real root.
Suppose $\deg(p_i) < d-1$ for all $i = 1, \dots, d-1$. 
Then $U'$ is a subspace of $\R[x]_{\leq d-2}$ of dimension $d-1$, so $U' = \R[x]_{\leq d-2}$ and $U'U' = \R[x]_{\leq 2d-4}$.
However, $p_d = x^d + q(x)$ for some $q \in \R[x]_{\leq d-1}$ and $x^{d-3} p_d, x^{d-2} p_d, p_d^2 \in UU$ are of degree $2d-3, 2d-2$ and $2d$, respectively, which would imply that $\dim(UU) \geq 2d$, a contradiction.
We conclude that $p_1, \dots, p_{d-1}$ have a common real root.
In particular, $1 \notin U'$ and since every polynomial in $U \setminus U'$ has degree $d$, also $1 \notin U$.
Finally, Lemma \ref{lem:max_face_dim_means_root} gives the desired conclusion.
\end{proof}

\begin{cor} \label{cor:no_face_of_max_dim}
Let $d \geq 3$. 
Let $f \in \R[x,y]_{2d}$ be positive on $\P^1(\R)$. 
If $F \subset \Gram(f)$ is a face of rank $d$, then $dim(F) < \choose{d-1}{2}$. \\ 
In other words: If $f \in \R[x,y]_{2d}$ is a nonnegative binary form and if $\Gram(f)$ contains a face $F$ of dimension $\choose{d-1}{2}$, then already $F = \Gram(f)$ (and $f$ has a real root).
\end{cor}
\begin{proof}
We have already seen that $\dim(F) \leq \choose{d-1}{2}$ (cf.~Corollary \ref{cor:max_dim_of_faces_real}). 
Choose $\theta \in \relint(F)$, $\theta = \sum_{i=1}^d p_i \otimes p_i$ with linearly independent $p_1, \dots, p_d \in \R[x,y]_d$.
If we had $\dim(F) = \choose{d-1}{2}$ we would have $\dim(UU) = 2d-1$ for $U = \scrU(F) = \im(\theta)$. 
But then $p_i(\xi) = 0$ for some $\xi \in \P^1(\R)$ ($i=1,\dots,d$) by Proposition \ref{prop:max_face_dim_means_root}, and thus $f(\xi) = p_1(\xi)^2 + \dots + p_d(\xi)^2 = 0$. 
\end{proof}

\begin{ex} \label{ex:dim_diag_real}
The following figure illustrates the possible combinations of $\rk(F)$ and $\dim(F)$ in the Gram spectrahedron of a positive binary form $f \in \Sigma_{2d}$ for $d = 5$ and $d = 8$.
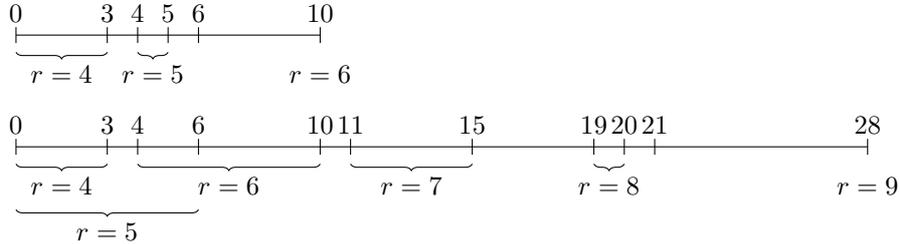
\vspace{0.5em}

\begin{tikzpicture}[decoration=brace]
    \draw(0,0)--(4,0);
    \foreach \x/\xtext in {0/$0$,1.2/$3$,1.6/$4$,2/$5$,2.4/$6$,4/$10$}
      \draw(\x,3pt)--(\x,-3pt) node[above=1ex] {\xtext};
    \draw[decorate, yshift=-1.5ex] (4,0) -- node[below=0.4ex] {$r=6$} (4,0);
	\draw[decorate, yshift=-1.5ex] (2,0) -- node[below=0.4ex] {$r=5$} (1.6,0);
	\draw[decorate, yshift=-1.5ex] (1.2,0) -- node[below=0.4ex] {$r=4$} (0,0);
\end{tikzpicture} 

\vspace{0.5em}

\begin{tikzpicture}[decoration=brace]
	\draw(0,0)--(11.2,0);
	\foreach \x/\xtext in {0/$0$,1.2/$3$,1.6/$4$,2.4/$6$,4/$10$,4.4/$11$,6/$15$,7.6/$19$,8/$20$,8.4/$21$,11.2/$28$}
      \draw(\x,3pt)--(\x,-3pt) node[above=1ex] {\xtext} ;
	\draw[decorate, yshift=-1.5ex] (11.2,0) -- node[below=0.4ex] {$r=9$} (11.2,0);  
	\draw[decorate, yshift=-1.5ex] (8,0) -- node[below=0.4ex] {$r=8$} (7.6,0);
	\draw[decorate, yshift=-1.5ex] (6,0) -- node[below=0.4ex] {$r=7$} (4.4,0);
    \draw[decorate, yshift=-1.5ex] (4,0) -- node[below=0.4ex] {$r=6$} (1.6,0);
    \draw[decorate, yshift=-5.5ex] (2.4,0) -- node[below=0.4ex] {$r=5$} (0,0);
    \draw[decorate, yshift=-1.5ex] (1.2,0) -- node[below=0.4ex] {$r=4$} (0,0);
\end{tikzpicture}
\end{ex}

%-------------------------------------------------------------------%

\section{Hermitian Gram spectrahedra} 
Chua, Plaumann, Sinn and Vinzant \cite{CPSV} also discuss the Hermitian analog to symmetric Gram spectrahedra. 
Similar to Scheiderer's method presented in \autoref{sec:real_Gram_spectrahedra} we pursue a coordinate-free approach to Hermitian Gram spectrahedra.
We also compare dimension bounds in the Hermitian case to those from the real symmetric case.

\begin{dfn}
Let $V$ and $V'$ be complex vector spaces. 
A mapping $\phi \colon V \to V'$ is said to be \emph{antilinear} if \[
\phi(v+w) = \phi(v) + \phi(w) \quad \text{and} \quad \phi(\lambda v) = \bar{\lambda} \phi(v)
\]
for all  $v, w \in V$ and all $\lambda \in \C$. 
An antilinear map $\phi \colon V \to V$ is called \emph{antilinear involution} if $\phi \circ \phi = \id_V$. 
\end{dfn}

\begin{lab}
Let $V$ be a vector space over $\R$. 
The \emph{complexification} $V_\C$ of $V$ is the tensor product of $V$ with the complex numbers, i.e.~$V_\C = V \otimes_\R \C$. 
Then $V_\C$ is a complex vector space and every $v \in V_\C$ can be written uniquely in the form \[
v = v_1 \otimes 1 + v_2 \otimes i
\]
where $v_1, v_2 \in V$.
We will drop the tensor product symbol and simply write $v = v_1 + i v_2$. 
We call $v_1$ the \emph{real part} and $v_2$ the \emph{imaginary part} of $v$.
Multiplication by the complex number $a+ib$ is given by the usual rule \[ 
(a+ib)(v_1+iv_2) = (av_1 - bv_2)+i(bv_1 + av_2).
\]
On $V_\C$ we have a natural antilinear involution $\phi \colon V_\C \to V_\C$ given by $\phi(v_1+iv_2) = v_1-iv_2$.
We also write $\ol{v} := \phi(v)$ for $v \in V_\C$.

If $U \subset V_\C$ is a subspace we define the \emph{complex conjugate} of $U$ to be the subspace $\ol{U} := \phi(U)$. 
We consider the tensor product $U \otimes_\C \ol{U}$ with the antilinear involution \[
U \otimes_\C \ol{U} \to U \otimes_\C \ol{U}, \quad 
v \otimes \ol{w} \mapsto w \otimes \ol{v} \quad (v, w \in U). 
\]
The fixed locus of this map is the real subspace of \emph{Hermitian tensors} in $U \otimes_\C \ol{U}$ which we will denote by $\sfH_2 U$.
\end{lab}

\begin{lab}
A hermitian tensor $\theta = \sum_j v_j \otimes \ol{w_j} \in \sfH_2 U$ ($v_j, w_j \in U$) can be identified with a Hermitian sesquilinear form $b_\theta \colon U^\du \times U^\du \to \C$ given by $b_\theta(\lambda, \mu) = \sum_j \lambda(v_j) \ol{\mu(w_j)}$. 
Here, $U^\du = \Hom(U, \C)$ denotes the dual space of $U$.
By the choice of a basis, $\theta$ resp.~$b_\theta$ can be identified with a Hermitian matrix $A$. 
We can diagonalize $\theta$, i.e.~we can write $\theta = \sum_{j=1}^r \epsilon_j (v_j \otimes \ol{v_j})$, with $r \geq 0$, $\epsilon_j = \pm 1$ and $v_1, \dots, v_r \in U$ linearly independent. 
In this case, the \emph{rank} of $\theta$ is $\rk(\theta) := r$, and $\theta$ is \emph{positive semidefinite} if all $\epsilon_j = 1$, or equivalently, $A \succeq 0$.
The \emph{image} of $\theta$ is $\im(\theta) := \spn(v_1, \dots, v_r)$. 
Note that if $u_1, \dots, u_N$ ($N = \dim_\C(U)$) is a basis of $U$ and $A \in \mathbb{H}_N$ is the Hermitian $N \times N$-matrix associated to $\theta$ with respect to this basis, then \[
\im(\theta) = \spn \left( \sum_{k=1}^N  (A e_j)_k u_k : j = 1, \dots, N \right). 
\]   
\end{lab}

\begin{lab}
Let $A$ be an $\R$-algebra and let $V \subset A_\C$ be a complex subspace of finite dimension. 
By $V \ol{V}$ we denote the subspace of $A_\C$ which is $\C$-linearly generated by the products $p \ol{q}$ ($p, q \in V$).
The multiplication map $\mu \colon V \times \ol{V} \to V \ol{V}$, $(p, \ol{q})\mapsto p \ol{q}$ is $\R$-bilinear and induces an $\R$-linear map $\sfH_2 V \to V \ol{V}$.
Given $f \in A$, the Hermitian tensors $\theta \in \sfH_2 V$ with $\mu(\theta) = f$ are called the \emph{Hermitian Gram tensors} of $f$, relative to $V$.
The \emph{Hermitian Gram spectrahedron} of $f$, relative to $V$, is the set of all positive semidefinite Hermitian Gram tensors of $f$ in $\sfH_2 V$, i.e. \[
\H^\plus(f) := \sfH_2^\plus V \cap \mu^{-1}(f).
\]
\end{lab}

We can adopt the definition of face subspaces from the real symmetric case to obtain a bijection between the nonempty faces of $\H^\plus(f)$ and the face subspaces $U \subset V$ for $\H^\plus(f)$.

\begin{dfn}
We say that a linear subspace $U$ of $V$ is facial, or a face subspace
(for the given spectrahedron $\H^\plus(f)$), if there exists $\theta \in \H^\plus(f)$ with $U = \im(\theta)$.
\end{dfn}

\begin{prop} 
For $U \subset V$ a face subspace for $f$, the face $\scrF(U)$ of $\H^\plus(f)$ has dimension \[
\dim \scrF(U) = \dim_\C(U)^2 -  \dim_\C(U \ol{U}).
\]
\end{prop}
\begin{proof} 
The dimension of the convex set $\scrF(U)$ is the dimension of the (real) affine space $\mu^{-1}(f) \cap \sfH_2 U$. 
Therefore, $\dim \scrF(U) = \dim_\R(W)$ where $W$ is the kernel of the $\R$-linear map $\mu \colon \sfH_2 U \to U \ol{U}$. 
The image of $\mu$ is $\im(\mu) = \spn_\R(p \ol{q} : p, q \in U)$. 
Hence, $\dim_\R (\im(\mu)) = \dim_\C (\im(\mu)_\C) = \dim_\C(U \ol{U})$.
So using that $\dim_\R(\sfH_2 U) = \dim_\C(U)^2$, the claim follows from the rank-nullity theorem.
\end{proof}

\begin{cor}
Let $f = \sum_{j=1}^r p_j \ol{p_j}$ with $p_1, \dots, p_r \in V$ linearly independent, let $\theta = \sum_{j=1}^r p_j \otimes \ol{p_j}$ be the corresponding Hermitian Gram tensor of $f$.
The dimension of the supporting face of $\theta$ in $\H^\plus(f)$ equals the number of independent linear relations between the products $p_j \ol{p_k}$ ($1 \leq j, k \leq r)$. \qed
\end{cor}

The Hermitian Gram spectrahedron of a positive binary form $f \in \R[x, y]_{2d}$ has dimension $d^2$. 
For the symmetric Gram spectrahedron we have $\dim \Gram(f) = \choose{d}{2}$ and in Corollary \ref{cor:max_dim_of_faces_real} we have shown that a face of $\Gram(f)$ of rank $r$ has dimension at most $\choose{r-1}{2}$.  
We get analogous bounds in the Hermitian case: 

\begin{cor}
Let $f \in \R[x, y]_{2d}$ be a nonnegative binary form.
Let $F \subset \H^\plus(f)$ be a face of rank $r$.
Then $\dim(F) \leq (r-1)^2$.
\end{cor}
\begin{proof}
Consider the subspace $U = \scrU(F)$ of $ \C[x,y]_d$. 
The same argument as in the real case shows that $\dim_\C(U \ol{U}) \geq 2r-1$ (cf.~Proposition \ref{prop:UU_geq_2r-1}).
Therefore, 
\[
\dim(F) = r^2 - \dim_\C(U \ol{U}) \leq r^2 - (2r-1) = (r-1)^2.  \qedhere
\]
\end{proof}

\begin{rem} 
Let $\rk(F) = r$. 
Writing $k = (d+1)-r$, i.e.~$k = \dim \ker(\theta)$ for $\theta \in \relint(F)$, we obtain
\[
(d-k)^2 - 2k \leq \dim(F) \leq (d-k)^2. 
\]
\end{rem}

\begin{ex}[cf.~Example \ref{ex:dim_diag_real}]
The following figure illustrates the possible combinations of $\rk(F)$ and $\dim(F)$ for faces of the Hermitian Gram spectrahedron of a positive binary form $f \in \Sigma_{2d}$ for $d = 5$ and $d = 8$.
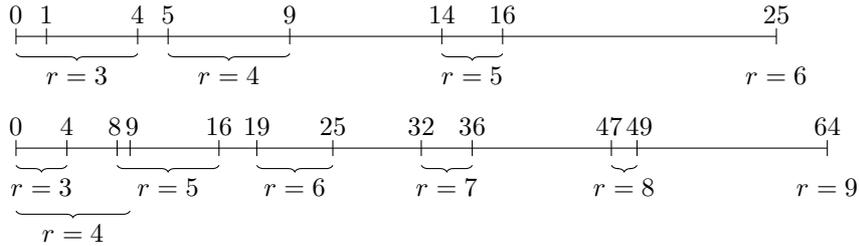
\vspace{0.5em}

\begin{tikzpicture}[decoration=brace]
	\draw(0,0)--(10,0);
	\foreach \x/\xtext in {0/$0$,0.4/$1$,1.6/$4$,2/$5$,3.6/$9$,5.6/$14$,6.4/$16$,10/$25$}
      \draw(\x,3pt)--(\x,-3pt) node[above=1ex] {\xtext};
	\draw[decorate, yshift=-1.5ex] (10,0) -- node[below=0.4ex] {$r=6$} (10,0);
	\draw[decorate, yshift=-1.5ex] (6.4,0) -- node[below=0.4ex] {$r=5$} (5.6,0);
	\draw[decorate, yshift=-1.5ex] (3.6,0) -- node[below=0.4ex] {$r=4$} (2,0);
    \draw[decorate, yshift=-1.5ex] (1.6,0) -- node[below=0.4ex] {$r=3$} (0,0);
\end{tikzpicture} 

\vspace{0.5em}

\begin{tikzpicture}[decoration=brace]
	\draw(0,0)--(10.67,0);
	\foreach \x/\xtext in {0/$0$,0.67/$4$,1.33/$8 \,$,1.5/$\, 9$,2.67/$16$,3.167/$19$,4.167/$25$,5.33/$32$,6/$36$,7.83/$47 \,$,8.167/$\, 49$,10.67/$64$}
      \draw(\x,3pt)--(\x,-3pt) node[above=1ex] {\xtext} ;
	\draw[decorate, yshift=-1.5ex] (10.67,0) -- node[below=0.4ex] {$r=9$} (10.67,0);  
	\draw[decorate, yshift=-1.5ex] (8.167,0) -- node[below=0.4ex] {$r=8$} (7.83,0);
	\draw[decorate, yshift=-1.5ex] (6,0) -- node[below=0.4ex] {$r=7$} (5.33,0);
    \draw[decorate, yshift=-1.5ex] (4.167,0) -- node[below=0.4ex] {$r=6$} (3.167,0);
    \draw[decorate, yshift=-1.5ex] (2.67,0) -- node[below=0.4ex] {$r=5$} (1.33,0);
    \draw[decorate, yshift=-5.5ex] (1.5,0) -- node[below=0.4ex] {$r=4$} (0,0);
    \draw[decorate, yshift=-1.5ex] (0.67,0) -- node[below=0.4ex] {$r=3$} (0,0);
\end{tikzpicture}
\end{ex}

We show how the concepts of Hermitian Gram tensors and facial subspaces can be used to give straightforward proofs of some facts presented in \cite{CPSV}. 
Subsequently, we will see that for generic $f \in \R[x, y]_{2d}$ the Hermitian Gram spectrahedron of $f$ contains a face of rank $d$ and dimension $(d-1)^2$.
In contrast to that, the symmetric Gram spectrahedron does \emph{not} contain a face of rank $d$ and dimension $\choose{d-1}{2}$, as we have shown in \ref{cor:no_face_of_max_dim}.

\begin{prop}[cf.~\cite{CPSV}, Prop.~5.6]
\label{ext:CPSV_factor_isomorphic_to_face}
If $f \in \R[\x]_{2P}$ factors as $f = g \ol{g} \cdot h$, where $g \in \C[\x]$ and $h \in \R[\x]$, then $\H^\plus(h)$ is linearly isomorphic to a face of $\H^\plus(f)$.
\end{prop}
Here, $P \subset \R^n$ is a polytope with vertices in $\Z_{\geq 0}^n$ and $\R[\x]_{2P}$ denotes the vector space of all polynomials $f \in \R[\x] = \R[x_1, \dots, x_n]$ whose Newton polytope $\Newt(f)$ is contained in $2P$.
\begin{proof}
At first, we follow the argumentation of \cite{CPSV}: 
Without loss of generality, we can assume that $2P$ equals the Newton polytope of f. 
Then $2P$ is the Minkowski sum of the polytopes $2 \Newt(g)$ and $\Newt(h)$. 
Therefore, we can write $\Newt(h)$ as $2Q$ for some $Q \subset \R^n$ with integer vertices. 
We see that $P$ is the Minkowski sum $\Newt(g) + Q$. \\  
For the rest of the proof, we will argue using Hermitian Gram tensors instead of matrices. 
We choose some $\theta \in \relint(\H^\plus(h))$, for instance $\theta = \sum_{j=1}^r p_j \otimes \ol{p_j}$, with $p_1, \dots, p_r \in \C[\x]_Q$ linearly independent and $\mu(\theta) = h$.
Let $U' = \im(\theta) = \spn(p_1, \dots, p_r) \subset \C[\x]_Q$ and consider $U := gU' \subset \C[\x]_P$.
Then $(g p_1, \dots, g p_r)$ is a basis of $g U' = U$ and 
\[
\mu \left( \sum_{j=1}^r g p_j \otimes \ol{g p_j} \right) = g\ol{g} \cdot \sum_{j=1}^r p_j \ol{p_j} = g \ol{g} \cdot h = f.
\]
Therefore, $U \subset \C[\x]_P$ is a facial subspace for $\H^\plus(f)$ and the face $\scrF(U)$ of $\H^\plus(f)$ is linearly isomorphic to $\H^\plus(h)$.
\end{proof}

\begin{cor}[cf.~\cite{CPSV}, Cor.~5.7]
Let $f \in \R[x, y]_{2d}$ be a positive binary form with distinct roots. 
Then $\H^\plus(f)$ contains $2^d$ tensors of rank one. 
The sum of rank-one tensors $p_1 \otimes \ol{p_1}, \dots, p_s \otimes \ol{p_s}$ in $\H^\plus(f)$ satisfies \[
\rk \left( \sum_{k=1}^s p_k \otimes \ol{p_k} \right) \leq d+1 - \deg( \gcd(p_1, \dots, p_s)).
\]
For each $2 \leq s \leq 2^d$, there are $s$ rank-one tensors in $\H^\plus(f)$ whose sum has rank at most $\lceil \log_2(s) \rceil + 1$. 
\end{cor}
\begin{proof}
Let $g = \gcd(p_1, \dots, p_r)$ and $\deg(g) = e$, for instance $p_k = g q_k$ for some $q_k \in \C[x, y]_{d-e}$ ($k = 1, \dots, s$). 
Then 
\begin{IEEEeqnarray*}{rCl}
\rk \left( \sum_{k=1}^s p_k \otimes \ol{p_k} \right) 
& = & \dim \spn(p_1, \dots, p_s) = \dim \spn(q_1, \dots, q_s) \\
& \leq & \dim \C[x, y]_{d-e} = (d-e)+1 = d+1 - \deg(g).
\end{IEEEeqnarray*}
If $f(x, 1)$ is monic with roots $a_1, \dots, a_d, \ol{a_1}, \dots, \ol{a_d} \in \C$, then the $2^d$ tensors of rank one in $\H^\plus(f)$ are exactly the tensors $p \otimes \ol{p}$ with $p = \prod_{j=1}^d (x - b_j y) \in \C[x, y]_d$, where $b_j \in \{ a_j, \ol{a_j} \}$. \\
For the last part of the claim let $2 \leq s \leq 2^d$ and $e \in \N_0$ such that $2^{d-e-1} < s \leq 2^{d-e}$. 
Consider $g = \prod_{j=1}^e (x - a_j y)$, and choose $s$ pairwise different elements $q_1, \dots, q_s$ of the set \[
\left\{ \prod_{j=e+1}^d (x - b_j y) : b_j \in \{ a_j, \ol{a_j} \}   \right\}
\]
(which has cardinality $2^{d-e}$). 
Let $p_j = g q_j$. 
Then $p_1 \otimes \ol{p_1}, \dots, p_s \otimes \ol{p_s}$ are rank-one Hermitian Gram tensors of $f$, and $g$ is the greatest common divisor of $p_1, \dots, p_s$. 
Hence, \[
\rk \left( \sum_{k=1}^s p_k \otimes \ol{p_k} \right) \leq d+1 - \deg(g) = (d-e)+1 = \lceil \log_2(s) \rceil + 1. \qedhere
\] 
\end{proof}

\begin{cor} 
Let $f \in \R[x, y]_{2d}$ be a positive binary form. 
For each $1 \leq r \leq d+1$ there is a face $F \subset \H^\plus(f)$ with $(\rk(F), \dim(F)) = (r, (r-1)^2)$.
\end{cor}
\begin{proof}
Fix $r \in \{1, \dots, d+1\}$.
Since a binary form factors completely into Hermitian squares, we can write $f = g \ol{g} \cdot h$, where $h \in \R[x, y]_{2(r-1)}$ is again a positive binary form and $g \in \C[x, y]_{d-(r-1)}$. 
By Proposition \ref{ext:CPSV_factor_isomorphic_to_face}, $\H^\plus(h)$ is linearly isomorphic to a face $F$ of $\H^\plus(f)$. 
Because $h$ is in the interior of $\Sigma_{2(r-1)}$, any tensor in the relative interior of $\H^\plus(h)$ has rank $r$, and $\dim(\H^\plus(h)) = (r-1)^2$.
\end{proof}

%-------------------------------------------------------------------%

\section{Polyhedral faces of Hermitian Gram spectrahedra}
Laurent and Poljak (see \cite{LP}) analyze the facial structure of elliptopes. 
They are also interested in the polyhedral faces of those spectrahedra. 
Some techniques presented in the proof of Theorem 4.1 in \cite{LP} turn out to be helpful for understanding polyhedral faces in Hermitian Gram spectrahedra:

\begin{thm} \label{thm:dim_bound_herm}
Let $f \in \R[x, y]_{2d}$ be a nonnegative binary form of degree $2d$.
Let $F$ be a polyhedral face of $\H^\plus(f)$ of dimension $k$. 
Then $\choose{k+1}{2} \leq d$.
Moreover, if all vertices of $F$ are rank-one tensors, then $F$ is a simplex. 
\end{thm}
\begin{proof}
Let $F_0 \subset F_1 \subset \dots \subset F_k := F$ be a chain of faces of $F$, where $\dim(F_j) = j$ for all $j$. 
We must have $\rk(F_j) \geq j+1$, so $r := \rk(F) \geq k+1$.
Therefore, \[
k = \dim(F) \geq r^2 - (2d+1) \geq (k+1)^2 - (2d+1),
\] 
which is equivalent to $\choose{k+1}{2} \leq d$. \\
Suppose now that all vertices of $F$ are rank-one tensors resp.~rank-one matrices, for instance $\Ex(F) = \{ v_j v_j^\ast : j \in J \}$.
Then $\scrU(F) = \spn(v_j : j \in J)$ and $\dim \scrU(F) = \rk(F) \geq k+1$.
Choose $k+1$ linearly independent vectors $v_0, \dots, v_k$ from $\{ v_j : j \in J \}$.
Then the vertices $v_j v_j^\ast$ ($j = 0, 1, \dots, k$) affinely span the polyhedron $F$. 
We show that those are the only vertices of $F$. 
Assume $X$ is another vertex of $F$. 
Then $X = \sum_{j=0}^k \alpha_j v_j v_j^\ast$ with $\sum_{j=0}^k \alpha_j = 1$. 
Let $l \in \{0, 1, \dots, k\}$ and let \[
0 \neq u \in \spn(v_j : j = 0, 1, \dots, k, j \neq l)^\perp \cap \spn(v_0, v_1, \dots, v_k).
\]
Since $X$ is psd, we obtain \[
0 \leq u^\ast X u = \sum_{j=0}^k \alpha_j u^\ast v_j v_j^\ast u = \alpha_l u^\ast v_l v_l^\ast u = \alpha_l |u^\ast v_l|^2.
\]
But $u^\ast v_l \neq 0$ and therefore $\alpha_l \geq 0$. 
This means that $X$ is contained in the convex hull of $v_0 v_0^\ast, \dots, v_k v_k^\ast$, a contradiction.
We conclude that $F$ is a simplex.
\end{proof}

\begin{cor}
If $F \subset \H^\plus(f)$ is a polyhedral face of dimension $k$ and all vertices of $F$ are rank-one tensors, then $F$ is a simplex with vertices $\theta_j = p_j \otimes \ol{p_j}$ ($j = 0, \dots, k$) and the linear relations between the products $p_j \ol{p_l}$ are generated by the $k$ obvious relations $p_0 \ol{p_0} = p_j \ol{p_j}$ for $j = 1, \dots, k$. \qed
\end{cor}

Let $d \in \N$. 
We aim to construct a positive binary form $f \in \R[x, y]_{2d}$ with distinct roots such that the following holds: 
For all $k \in \N$ with $\choose{k+1}{2} \leq d$, there is a polyhedral face of dimension $k$ in the Hermitian Gram spectrahedron of $f$. 
We will need some preliminary work. 
[Note that \ref{prop:only_rel_among_herm_squares_means_polyhedral} - \ref{cor:max_rk_min_dim_sum_polyhedral} do not require binary forms.]

\begin{prop} \label{prop:only_rel_among_herm_squares_means_polyhedral}
Let $F$ be a face of $\H^\plus(f)$, $\rk(F) = r$. 
If there is a basis $p_1, \dots, p_r$ of $U = \scrU(F)$ such that $f = p_1 \ol{p_1} + \dots + p_r \ol{p_r}$ and the linear relations among the products $p_j \ol{p_k}$ ($1 \leq j, k \leq r$) only involve the Hermitian squares $p_1 \ol{p_1} \dots, p_r \ol{p_r}$, then $F$ is polyhedral. 
\end{prop}
\begin{proof}
We write $\mathfrak{p} = (p_1, \dots, p_r)^T$.
With respect to the basis $p_1, \dots, p_r$ of $U$, the Gram matrices of $f$ relative to $U$ are of the form $I_r + A$, where $A$ is a Hermitian $r \times r$-matrix with $\mathfrak{p}^T A \ol{\mathfrak{p}} = 0$.
For $j \neq k$ there is no linear relation among the generators of $U \ol{U}$ which involves $p_j \ol{p_k}$.
Therefore, any such $A$ is diagonal.
Thus, for this basis of $U$, the elements of $F$ correspond to the solutions of a diagonal linear matrix inequality, so $F$ is polyhedral. 
\end{proof}

\begin{rem}
In the situation of the preceding proposition, any $\theta \in F$ has a representation $\theta = \sum_{j=1}^r a_j^2 (p_j \otimes \ol{p_j})$ with $a_j \in \R$. 
Indeed, if $D$ is the (diagonal) Hermitian Gram matrix associated to $\theta$ with respect to the basis $p_1, \dots, p_r$ of $U$, then we can choose $a_j$ to be a square root of $D_{jj} \in \R_{\geq 0}$ ($j = 1, \dots, r$).
\end{rem}

\begin{cor} \label{cor:max_rk_min_dim_sum_polyhedral}
Let $\theta_0, \dots, \theta_k \in \H^\plus(f)$.
Let $F := \suppface(\theta_i : i = 0, \dots, k) = \suppface \left( \frac{1}{k+1} (\theta_0 + \dots + \theta_k) \right)$.
If $\dim(F) = k$ and $\rk(F) = \sum_i \rk(\theta_i)$, then $F$ is polyhedral.
\end{cor}
\begin{proof}
We write $\theta_i = \sum_{j=1}^{r_i} p_j^{(i)} \otimes \ol{p_j^{(i)}}$, where $r_i = \rk(\theta_i)$. 
By the assumption on the rank of $F$, $(p_j^{(i)} : 1 \leq i \leq k, 1 \leq j \leq r_i)$ is linearly independent.
Now we have at least the $k$ relations \[
\sum_{j=1}^{r_0} p_j^{(0)} \ol{p_j^{(0)}} = \sum_{j=1}^{r_i}  p_j^{(i)} \ol{p_j^{(i)}} \quad (i = 1, \dots, k).
\] 
Because of $\dim(F) = k$, there are no further independent relations. 
Hence, $F$ is polyhedral (see Proposition \ref{prop:only_rel_among_herm_squares_means_polyhedral}).
\end{proof}

\begin{rem} \label{rem:max_rk_min_dim_binary_forms_herm}
In particular, if $f$ is a binary form and if the supporting face $F$ of $k+1$ rank-one extreme points of $\H^\plus(f)$ has rank $k+1$ (which is the maximal possible rank of $F$) and dimension $k$ (which is the minimal possible dimension in this situation), then $F$ is polyhedral. \\
These conditions also imply that all vertices of $F$ are rank-one tensors. 
Indeed, if $F_0 \subset F_1 \subset \dots \subset F_k := F$ was a chain of faces of $F$, where  $\rk(F_0) \geq 2$ and $\dim(F_j) = j$ for all $j$, we would have $\rk(F) \geq k+2$, a contradiction.
So by Theorem \ref{thm:dim_bound_herm}, $F$ is even a simplex.
\end{rem}

In our construction of polyhedral faces, we will often be in the situation that we want to find a form $s \in \C[x, y]_k$ which does not divide any nonzero element of a given subspace $U \subset \C[x, y]_d$ of dimension $k = \deg(s)$.

\begin{prop} \label{prop:construction_special_basis}
Let $U \subset \C[x]_{\leq d}$ be a linear subspace of dimension $\dim(U) = k \leq d$.
Then there are $\lambda_1, \dots, \lambda_k \in \C$ and a basis $p_1, \dots, p_k$ of $U$ such that $p_l(\lambda_l) \neq 0$ and $p_j(\lambda_l) = 0$ for all $j > l$. 
In particular, whenever $p \in U$ vanishes in $\lambda_1, \dots, \lambda_k$ then $p = 0$.
\end{prop}
\begin{proof}
We use the following inductive procedure to find scalars $\lambda_1, \dots, \lambda_k \in \C$ and to construct a basis of $U$ with the desired properties.  
Start with any basis $q_1^{(1)}, \dots, q_k^{(1)}$ of $U$.
If $l \in \{1, \dots, k-1\}$, choose $\lambda_l \in \C \setminus \mathcal{V}(q_l^{(l)})$. 
Then set  \[
q_j^{(l+1)} := q_l^{(l)}(\lambda_l) \cdot q_j^{(l)} - q_j^{(l)}(\lambda_l) \cdot q_l^{(l)}, \quad j = l+1, \dots, k.
\]
This guarantees that $q_l^{(l)}(\lambda_l) \neq 0$ and $q_j^{(l+1)}(\lambda_l) = 0$ for all $j \geq l+1$. 
Furthermore, $q_1^{(1)}, q_2^{(2)}, \dots, q_l^{(l)}, q_{l+1}^{(l+1)}, \dots, q_k^{(l+1)}$ is again a basis of $U$. 
Also, if $\zeta \in \C$ was a common root of the polynomials $q_l^{(l)}, \dots, q_k^{(l)}$ in the step before (as is the case for $\lambda_1, \dots, \lambda_{l-1}$), it still is a common root of $q_l^{(l)}, q_{l+1}^{(l+1)}, \dots, q_k^{(l+1)}$. 
By construction, setting $p_l := q_l^{(l)}$ ($l = 1, \dots, k$) yields the desired basis. \\
Suppose $p \in U$ with $p(\lambda_l) = 0$ for $l = 1, \dots, k$. 
We have $p = \sum_{j=1}^k \mu_j p_j$ for some $\mu_j \in \C$ and \[
0 = p(\lambda_1) = \sum_{j=1}^k \mu_j p_j(\lambda_1) = \mu_1 p_1(\lambda_1).
\] 
Since $p_1(\lambda_1) \neq 0$ we deduce $\mu_1 = 0$.
Iterating this argument, one successively shows that all $\mu_j$'s are zero.
\end{proof}

\begin{rem} \label{rem:choose_scalars_generically}
We see from the proof of Proposition \ref{prop:construction_special_basis} that (for a fixed subspace $U$) the scalars $\lambda_1, \dots, \lambda_k \in \C$ can be chosen from an open dense subset of $\C^k$. 
Indeed: 
We start with an arbitrary basis $p_1^{(1)}, \dots, p_k^{(1)} \in \C[x]$ of $U$. 
If $l \in \{1, \dots, k-1\}$, we set 
\begin{IEEEeqnarray*}{rCl}
p_j^{(l+1)}(x_1, \dots, x_l, x_{l+1}) 
& := & 
p_l^{(l)}(x_1, \dots, x_{l-1}, x_l) \cdot p_j^{(l)}(x_1, \dots, x_{l-1}, x_{l+1})\\
&& -\> p_j^{(l)}(x_1, \dots, x_{l-1}, x_l) \cdot p_l^{(l)}(x_1, \dots, x_{l-1}, x_{l+1}) \\
& \in & \C[x_1, \dots, x_l, x_{l+1}],
\end{IEEEeqnarray*}
for all $j = l+1, \dots, k$.
For example, \[
p_j^{(2)}(x_1, x_2) =  p_1^{(1)}(x_1) \cdot p_j^{(1)}(x_2) - p_j^{(1)}(x_1) \cdot p_1^{(1)}(x_2).
\]
So if $p_1^{(1)}(\lambda_1) \neq 0$ then $p_j^{(2)}(\lambda_1, x) \in \C[x]$ coincides with $q_j^{(2)} \in \C[x]$ from the construction in Proposition \ref{prop:construction_special_basis}. 
We set $q_j := p_j^{(j)} \in \C[x_1, \dots, x_j] \subset \C[x_1, \dots, x_k]$.
Using induction, one can show that the set of suitable scalars contains 
\[
\left\{ (\lambda_1, \dots, \lambda_k) \in \C^k : \prod_{j=1}^k q_j(\lambda_1, \dots, \lambda_k) \neq 0 \right\}.
\]
This set is nonempty and open in the Zariski topology on $\C^k$, and therefore also open and dense in the Euclidean topology. \\
The set of all $(\lambda_1, \dots, \lambda_k) \in \C^k$ with $\lambda_{j'} \neq \lambda_j, \ol{\lambda_j}$ for all $j' \neq j$ is open and dense as well. 
Therefore, we will assume that 
$\lvert \{ \lambda_j, \ol{\lambda_j} : j = 1, \dots, k \} \rvert = 2k$ 
whenever needed.
\end{rem} 

We find polyhedral faces using the rank-one extreme points of $\H^\plus(f)$.
The set of these points is denoted by $\Ex_1(\H^\plus(f))$.

\begin{thm} \label{thm:construction_polyhedral_faces_hermitian}
Let $k \in \N$ and $d = \choose{k+1}{2}$. 
Then there exists a positive binary form $f \in \R[x, y]_{2d}$ with distinct roots such that $\H^\plus(f)$ contains a simplex face $F$ 
with $(\rk(F), \dim(F)) = (k+1, k)$ and $\Ex(F) \subset \Ex_1(\H^\plus (f))$.
\end{thm}
\begin{proof}
We proceed by induction on $k$. 
Let $k = 1$, so $d = 1$. 
If $f \in \R[x, y]_2$ is positive, then $\H^\plus(f)$ is an interval of rank $d+1 = 2 = k+1$ whose extreme points have rank one.
Now assume that $k \geq 2$, $d' = \choose{k}{2}$ and that we have a positive binary form $g \in \R[x, y]_{2d'}$ with distinct roots such that $\H^\plus(g)$ contains a polyhedral face $F'$ with $(\rk(F'), \dim(F')) = (k, k-1)$ and $\Ex(F') \subset \Ex_1(\H^\plus(g))$.
Then $U' := \scrU(F')$ is spanned by some linearly independent $p_1, \dots, p_k \in \C[x, y]_{d'}$ with $g = p_j \ol{p_j}$ ($j = 1, \dots, k$). 
Moreover, $\dim_\C(U') = k$ and $U' \ol{U'} = \C[x, y]_{2d'}$.
Let $\alpha_1, \dots, \alpha_{d'}, \ol{\alpha_1} \dots, \ol{\alpha_{d'}} \in \C$ denote the (distinct) roots of $g(x, 1)$. 
By Proposition \ref{prop:construction_special_basis}, we find $\beta_1, \dots, \beta_k \in \C$ such that \[ 
\lvert \{ \alpha_j, \ol{\alpha_j} : j = 1, \dots, d' \} \cup \{ \beta_j, \ol{\beta_j} : j = 1, \dots, k \} \rvert = 2d' + 2k = 2 \choose{k+1}{2} = 2d,
\]
and that whenever $p \in U'$ vanishes in $(\beta_1 : 1), \dots, (\beta_k : 1) \in \P^1$ then $p = 0$. 
We define $s := \prod_{j=1}^k (x-\beta_j y) \in \C[x, y]_k$, 
and $t := p_l$ for some $l \in \{1, \dots, k\}$, e.g.~$t=p_1$. 
Then $f := (st) \ol{(st)} = s \ol{s} \cdot g \in \R[x, y]_{2d}$ has distinct roots.
Consider $F := \suppface(\theta_0, \dots, \theta_k) \subset \H^\plus(f)$, where $\theta_0 = st \otimes \ol{st}$ and $\theta_j = \ol{s}p_j \otimes s \ol{p_j}$ for $j = 1, \dots, k$.
We have to show that $(\rk(F), \dim(F)) = (k+1, k)$ (that $F$ is a simplex is then clear by \ref{rem:max_rk_min_dim_binary_forms_herm}).
The rank of $F$ is given by the dimension of the subspace \[
U := \scrU(F) = \spn(st, \ol{s}p_1, \dots, \ol{s}p_k) = \C \cdot st + \ol{s}U'.
\]
If $\alpha st = \ol{s}p$ for some $\alpha \in \C$ and $p \in U'$, then $s$ divides $p$ since $s$ and $\ol{s}$ are coprime. 
Therefore, $p(\beta_j, 1) = 0$ for all $j$, which implies $p = 0$. 
We conclude that $U = \C \cdot st \oplus \ol{s}U'$ and $\rk(F) = \dim_\C(U) = \dim_\C(U') + 1 = k+1$.
Furthermore, 
\begin{IEEEeqnarray*}{rCl}
U \ol{U} & = & s \ol{s} \cdot U' \ol{U'} + s^2 t \ol{U'} + \ol{s^2 t} U' + \C \cdot \underbrace{(st)\ol{(st)}}_{\in s \ol{s} \cdot U' \ol{U'}} \\ 
& = & s \ol{s} \cdot \C[x, y]_{2d'} + s^2 t \ol{U'} + \ol{s^2 t} U'.
\end{IEEEeqnarray*}
We first show that $s^2 t \ol{U'} \cap \ol{s^2 t} U' = \{0\}$. 
Let $u, v \in U'$ such that $s^2 t \ol{u} = \ol{s^2 t} v$. 
Since $s^2 t$ and $\ol{s^2 t}$ have no roots in common, this implies that $s^2 t$ divides $v$. 
So if $v \neq 0$ we would get \[
\deg(v) \geq \deg(s^2 t) = 2 \deg(s) + \deg(t) > \deg(t) = \deg(v).
\] 
Therefore, the sum $s^2 t \ol{U'} + \ol{s^2 t} U'$ is direct. 
Now we show that also
\begin{equation} \label{eq:construction_polyhedral_faces_hermitian:directsum}
(s \ol{s} \cdot \C[x, y]_{2d'}) \cap (s^2 t \ol{U'} \oplus \ol{s^2 t} U') = \{0\}.   \tag{$\ast$} 
\end{equation}
Suppose that $s \ol{s} q = s^2 t \ol{u} + \ol{s^2 t} v$, where $q \in \C[x, y]_{2d'}$ and $u, v \in U'$. 
Then \[
s(\ol{s} q - st\ol{u}) = \ol{s^2 t} v,
\]
and hence $s$ divides $v$. 
But since $v$ is in $U'$, $v = 0$ by the choice of $s$. 
Analogously, we see that $u = 0$. 
This proofs (\ref{eq:construction_polyhedral_faces_hermitian:directsum}). 
To sum up, 
\begin{IEEEeqnarray*}{rCl}
\dim_\C(U \ol{U}) & = & \dim_\C( \C[x, y]_{2d'} ) + 2 \dim_\C(U') \\
& = & 2d' + 1 + 2k \\
& = & 2(d'+k)+1 \\
& = & 2d+1,
\end{IEEEeqnarray*}
and therefore $\dim(F) = (k+1)^2 - 2 \choose{k+1}{2} - 1 = k$.
\end{proof}

\begin{dfn}
Let $k \in \N$ and $d = \choose{k+1}{2}$. 
We define $P_{2d}$ to be the set of all $f \in \R[x, y]_{2d}$ such that $f \in \interior(\Sigma_{2d})$ has distinct roots and $\H^\plus(f)$ contains a simplex face $F$ with $(\rk(F), \dim(F)) = (k+1, k)$ and $\Ex(F) \subset \Ex_1(\H^\plus(f))$. 
\end{dfn}

\begin{lab}
Consider the set $W$ of all tuples $(f, p_0, p_1, \dots, p_k) \in \R[x, y]_{2d} \times \C[x, y]_d^{k+1}$ such that the following hold: 
$f \in \interior(\Sigma_{2d})$ has distinct roots (i.e.~the discriminant $D(f(x,1)) \neq 0$), 
$f = p_0 \ol{p_0} = \dots = p_k \ol{p_k}$, 
and for $U = \spn_\C(p_0, \dots, p_k) \subset \C[x, y]_d$ we have $\dim_\C(U) = k+1$ and $U \ol{U} = \C[x, y]_{2d}$. 
Separating real and imaginary part of the coefficients, we can express all these conditions by polynomial equations and inequalities (over $\R$) in those coefficients. 
For instance, $U \ol{U} = \C[x, y]_{2d}$ if and only if not all $(2d+1)$-minors of the matrix containing the coefficients of $p_j \ol{p_{j'}}$ ($0 \leq j, j' \leq k$) vanish. 
These minors are polynomial expressions in the real and imaginary parts of the coefficients of $p_0, \dots, p_k$.
Therefore, $W$ can be seen as an $\R$-semialgebraic set. 
By Remark \ref{rem:max_rk_min_dim_binary_forms_herm}, $P_{2d}$ is the projection of $W$ onto the first component, and hence semialgebraic itself.
\end{lab}

\begin{thm} \label{thm:generic_herm_case}
Let $k \in \N$ and $d \geq \choose{k+1}{2}$. 
The Hermitian Gram spectrahedron of a generic nonnegative binary form $f \in \R[x, y]_{2d}$ contains a simplex face $F$ with $(\rk(F), \dim(F)) = (k+1, k)$ and $\Ex(F) \subset \Ex_1(\H^\plus(f))$.
\end{thm}
\begin{proof}
It suffices to prove the theorem for $d = \choose{k+1}{2}$.
We have to show that $\Sigma_{2d} \setminus P_{2d}$ is contained in some hypersurface of $\R[x, y]_{2d}$. 
Since $P_{2d}$ is semialgebraic, by Proposition 2.8.13 in \cite{BCR} we have \[
\dim(\ol{P_{2d}} \setminus P_{2d}) < \dim(P_{2d}), 
\]
where $\ol{P_{2d}}$ is the Euclidean closure of $P_{2d}$. 
In particular, we will get that $\Sigma_{2d} \setminus P_{2d}$ is contained in some hypersurface if we can show that $P_{2d}$ is dense in $\Sigma_{2d}$. \\
We will proof this by induction on $k$. 
For $k = 1$ (i.e.~$d = 1$) the Hermitian Gram spectrahedron of any positive $f \in \R[x, y]_2$ is an interval of rank $2$ whose extreme points have rank one. 
Therefore, $P_1 = \interior(\Sigma_2)$. 
Now let $k \geq 2$, $d = \choose{k+1}{2}$, and define $d' = \choose{k}{2} = d - k$.
Let $h \in \interior(\Sigma_{2d})$ and let $U \subset \interior(\Sigma_{2d})$ be an open neighborhood of $h$. 
We have to show that $U$ contains some $\tilde{h} \in P_{2d}$. 
Consider the multiplication map \[
\psi \colon \R[x, y]_{2d'} \times \R[x, y]_{2k} \to \R[x, y]_{2d}, \quad (f, g) \mapsto fg.
\]
Without loss of generality, we can assume that the coefficient of $x^{2d}$ in $h$ is equal to $1$. 
Since $h$ is positive on $\P^1(\R)$ and $\deg(h) = 2d = 2(d'+k)$, we can write \[
h = \underbrace{ \prod_{j=1}^{d'} (x-\alpha_j y)(x-\ol{\alpha_j}y) }_{=: f \in \interior(\Sigma_{2d'})} 
\cdot 
\underbrace{ \prod_{j=1}^k (x-\beta_j y)(x-\ol{\beta_j}y) }_{=: g \in \interior(\Sigma_{2k})},
\] 
so that $h = \psi(f, g)$. 
The fact that $\psi$ is continuous implies that $\psi^{-1}(U)$ is open. 
In particular, there are open neighborhoods $V \subset \interior(\Sigma_{2d'})$ of $f$ and $W \subset \interior(\Sigma_{2k})$ of $g$, respectively, such that $\psi(V \times W) \subset U$.  
By induction, $P_{2d'} \subset \Sigma_{2d'}$ is dense in the Euclidean topology.
Therefore, we find a polynomial $\tilde{f} \in P_{2d'} \cap V$. 
Now, the set of suitable cofactors for $\tilde{f}$ is dense in $\Sigma_{2k}$ (see the proof of Theorem \ref{thm:construction_polyhedral_faces_hermitian} and Remark \ref{rem:choose_scalars_generically}).
This means that there is $\tilde{g} \in W$ such that $\tilde{h} := \tilde{f} \tilde{g} \in P_{2d}$.
Moreover, $\tilde{h} = \psi(\tilde{f}, \tilde{g}) \in \psi(V \times W) \subset U$. 
We conclude that $\ol{P_{2d}} = \Sigma_{2d}$, and this completes the proof.
\end{proof}

%-------------------------------------------------------------------%

\section{Polyhedral faces of Gram spectrahedra}
\label{sec:polyhedral_faces_real}
We will use the construction from the Hermitian case to construct polyhedral faces in the (symmetric) Gram spectrahedron of some binary forms. 

\begin{prop} \label{prop:dim_bound_polyhedral_faces_real}
Let $f \in \R[x, y]_{2d}$ be a generic nonnegative binary form of degree $2d$.  
Let $F \subsetneq \Gram(f)$ be a polyhedral face of dimension $k \geq 1$. 
Then $\choose{k+3}{2} \leq 2d-2$.
\end{prop}
\begin{proof}
There is a chain $F_0 \subset F_1 \subset \dots \subset F_k := F$ of faces of $F$, where $\dim(F_i) = i$, $i = 0, \dots, k$. 
The corresponding chain of ranks of these faces has to be strictly increasing. 
Since $f$ is generic, $\Gram(f)$ does not contain neither points of rank $1$ nor positive-dimensional faces of rank $3$. 
Therefore, $\rk(F_1) \geq 4$ and $r = \rk(F) \geq k + 3$.
So the estimation we get in the real symmetric case is \[
k = \dim(F) \geq \choose{r+1}{2} - (2d+1) \geq \choose{k+4}{2} - (2d+1),
\]
which is equivalent to the inequality in the claim.
\end{proof}

\begin{prop}
Let $F$ be a face of $\Gram(f)$, $\rk(F) = r$. 
If there is a basis $p_1, \dots, p_r$ of $U = \scrU(F)$ such that $f = p_1^2 + \dots + p_r^2$ and the quadratic relations among the $p_i$'s only involve the squares $p_1^2, \dots, p_r^2$, then $F$ is polyhedral. 
\end{prop}
\begin{proof}
Note that $p_j = \ol{p_j}$ for all $j$, and that $UU$ is generated by the products $p_j p_k$, where $1 \leq j \leq k \leq r$.
Aside from these modifications, the proof is the same as for Proposition \ref{prop:only_rel_among_herm_squares_means_polyhedral}.
\end{proof}

\begin{rem} \label{rem:diagonal_representation}
In the situation of the preceding proposition, any $\theta \in F$ has a representation $\theta = \sum_{i=1}^r (a_i p_i) \otimes (a_i p_i)$ with $a_i \in \R$. 
Indeed, if $D$ is the (diagonal) Gram matrix associated to $\theta$ with respect to the basis $p_1, \dots, p_r$ of $U$, then we can choose $a_i$ to be a square root of $D_{ii} \geq 0$ ($i = 1, \dots, r$).
\end{rem}

\begin{cor}[cf.~Cor.~\ref{cor:max_rk_min_dim_sum_polyhedral}]
Let $F \subset \Gram(f)$ be the supporting face of $k+1$ points $\theta_0, \dots, \theta_k \in \Gram(f)$. 
If $\dim(F) = k$ and $\rk(F) = \sum_i \rk(\theta_i)$, then $F$ is polyhedral. \qed
\end{cor}

\begin{rem} \label{rem:max_rk_min_dim_binary_forms_real}
In particular, if $f \in \Sigma_{2d}$ is a binary form and if the supporting face $F$ of $k+1$ rank-two extreme points of $\Gram(f)$ has rank $2(k+1)$ (which is the maximal possible rank of this face) and dimension $k$ (which is the minimal possible dimension in this situation), then $F$ is polyhedral. 
\end{rem}

To have a short notation, we write $\Ex_2(f)$ for the set of rank-two extreme points of $\Gram(f)$.

\begin{prop} \label{prop:even_simplex}
In the situation of Remark \ref{rem:max_rk_min_dim_binary_forms_real}, $F$ is a simplex, all vertices of $F$ are rank-two tensors, and the rank of any $\theta \in F$ is even.
\end{prop}
\begin{proof}
By assumption, $F$ is the supporting face of $k+1$ rank-two extreme points $\theta_0, \dots, \theta_k \in \Ex_2(f)$.
We write $\theta_i = q_{2i+1} \otimes q_{2i+1} + q_{2i+2} \otimes q_{2i+2}$ for $i = 0, \dots, k$. 
We know that $\mathcal{B} = (q_j : j = 1, \dots, 2(k+1))$ is a basis of $U = \scrU(F) \subset \R[x, y]_d$. 
Furthermore, the quadratic relations between the elements of $\mathcal{B}$ are generated by  \[
q_1^2 + q_2^2 = q_{2i+1}^2 + q_{2i+2}^2 \quad  (i = 1, \dots, k).
\]
So any quadratic relation is of the form
\[
\sum_{i=0}^k \lambda_i \left( q_{2i+1}^2 + q_{2i+2}^2  \right) = 0,
\]
where $\lambda_0, \lambda_1, \dots, \lambda_k \in \R$ with $\sum \lambda_i = 0$.
Let $\theta \in F$.
By Remark \ref{rem:diagonal_representation}, there is a representation $\theta = \sum_{j=1}^{2(k+1)} (a_j q_j) \otimes (a_j q_j)$ with $a_j \in \R$. 
This leads to the quadratic relation \[
q_1^2 + q_2^2 = f = \mu(\theta) = \sum_{i=0}^k a_{2i+1}^2 q_{2i+1}^2 + a_{2i+2}^2 q_{2i+2}^2. 
\]
Comparing this to the general appearance of a quadratic relation, we get $a_{2i+1}^2 = a_{2i+2}^2$ for all $i = 0, \dots, k$. 
In particular, $\rk(\theta) = 2 \cdot \vert \{i \in \{0, \dots, k\} : a_{2i+1} \neq 0 \} \vert$ is even. 
Besides, we can rewrite \[
\theta = \sum_{i=0}^k b_i^2 (q_{2i+1} \otimes q_{2i+1} + q_{2i+2} \otimes q_{2i+2}), 
\]
where $b_i = a_{2i+1}$.
Choose $l \in \{0, \dots, k\}$ with $b_l \neq 0$. 
Then \[
\scrU(\{\theta_l\}) = \spn(q_{2l+1}, q_{2l+2}) \subset \im(\theta) 
\]
and therefore $\theta_l \in \suppface(\theta)$. 
Hence, if $\theta \in \Ex(F)$ is an extreme point of $\Gram(f)$, then $\theta = \theta_l$. 
We conclude that $F = \conv(\Ex(F)) = \conv(\theta_0, \dots, \theta_k)$ is a simplex.
\end{proof}

\begin{rem}
For $d \leq 5$, Proposition \ref{prop:dim_bound_polyhedral_faces_real} implies $k \leq 1$, i.e.~there are no polyhedral faces bigger than edges. 
We are able to construct polynomials with two-dimensional polyhedral faces in their Gram spectrahedra as soon as $d \geq 9$ (see Theorem \ref{thm:construction_polyhedral_faces_real}). 
Such examples could in principle also exist for $d \in \{6, 7, 8\}$. 
But the rank of such a face $F$ would have to be (at most) $5$. 
This means that there have to be linear dependencies between the polynomials of which the different extreme points of $F$ are made up.
Furthermore, if $\Ex(F) \subset \Ex_2(f)$, then the face $F$ is not 'diagonalizable' as the following proposition shows.
\end{rem}

\begin{prop} \label{prop:diagonalizable_face_dim}
Let $f \in \R[x, y]_{2d}$ be a generic nonnegative binary form of degree $2d$. 
Let $F$ be a (polyhedral) face of dimension $k$ with $\Ex(F) \subset \Ex_2(f)$. 
If there is a basis $\mathcal{B}$ of $U = \scrU(F)$ such that all matrices associated to the Gram tensors in $F$ are diagonal with respect to $\mathcal{B}$, then $r = \rk(F) \geq 2(k+1)$ and $(k+1)^2 \leq d$. 
\end{prop}
\begin{proof}
Let $\mathcal{B} = (p_1, \dots, p_r)$ be such a basis of $U$. 
Since $F$ is a $k$-dimensional polytope and $\Ex(F) \subset \Ex_2(f)$, we can choose $k+1$ distinct extreme points $\theta_0, \dots, \theta_k \in \Ex(F)$ of rank two. 
For $i = 0, \dots, k$ we consider the subspaces $U_i = \scrU(\{\theta_i\}) \subset U$. 
By assumption, every $\theta \in F$ has a representation $\sum_{i=1}^r (a_i p_i) \otimes (a_i p_i)$ with $a_i \in \R$. 
This means that every facial subspace (for the given spectrahedron $F = \Gram_U(f)$) of $U$ is generated by a subset of the $p_i$'s. 
In particular, each $U_i$ has a basis consisting of exactly two elements out of $p_1, \dots, p_r$.
Assume that $r = \rk(F) < 2(k+1)$. 
Then there are $j \in \{1, \dots, r\}$ and $i_1 \neq i_2$ such that $p_j \in U_{i_1} \cap U_{i_2} \neq \{0\}$.
But then \[
2 < \rk(\theta_{i_1} + \theta_{i_2}) = \dim(U_{i_1} + U_{i_2}) \leq 3.
\]
This would mean that $\theta_{i_1}$ and $\theta_{i_2}$ are contained in a positive-dimensional face of rank three.
However, there is no such face in $\Gram(f)$ because $f$ is generic. \\
Finally, 
\[
k = \dim(F) \geq \choose{r+1}{2} - (2d+1) \geq \choose{2(k+1)+1}{2} - (2d+1),
\]
which can be simplified to $(k+1)^2 \leq d$.
\end{proof}

We will use the construction from the Hermitian case as a foundation for our construction in the real symmetric case. 
Since we aim for $k$-simpleces with extreme points of rank two (instead of rank one) we have to allow for higher degree polynomials (cf.~Proposition \ref{prop:diagonalizable_face_dim}), meaning that we will have to introduce another cofactor to make up for the difference.  
Besides, we have to make sure that not only $U \ol{U}$ is big (as needed in the Hermitian case), but also $UU$. 
This means that we have to choose the cofactor slightly more carefully.

\begin{lem} [cf.~Prop.~\ref{prop:construction_special_basis}]
\label{lem:careful_construction_special_basis} 
Let $U \subset \C[x]_{\leq d}$ be a linear subspace of dimension $\dim(U) = k \leq d$. 
Then there are $\lambda_1, \dots, \lambda_k \in \C$ such that the following holds: 
Whenever $p \in U$ and $p(\lambda_j) = 0$ for all $j$ \emph{or}  $p(\ol{\lambda_j}) = 0$ for all $j$, then $p = 0$.
\end{lem}
\begin{proof}
We start with \emph{two} arbitrary bases $q_1^{(1)}, \dots, q_k^{(1)}$ and $\tilde{q}_1^{(1)}, \dots, \tilde{q}_k^{(1)}$ of $U$.
If $l \in \{1, \dots, k-1\}$, we choose \[
\lambda_l \in \C \setminus \left( \mathcal{V} \left(q_l^{(l)} \right) \cup \mathcal{V} \left( \ol{\tilde{q}_l^{(l)}} \right) \right). 
\] 
Then set  
\begin{IEEEeqnarray*}{rCll}
q_j^{(l+1)} & := & q_l^{(l)}(\lambda_l) \cdot q_j^{(l)} - q_j^{(l)}(\lambda_l) \cdot q_l^{(l)} & \quad \text{and} \\
\tilde{q}_j^{(l+1)} & := & \tilde{q}_l^{(l)}(\ol{\lambda_l}) \cdot \tilde{q}_j^{(l)} - \tilde{q}_j^{(l)}(\ol{\lambda_l}) \cdot \tilde{q}_l^{(l)}, 
& \quad j = l+1, \dots, k.
\end{IEEEeqnarray*}
Among other things, this guarantees that $q_j^{(l+1)}(\lambda_l) = 0$ and $\tilde{q}_j^{(l+1)}(\ol{\lambda_l}) = 0$ for all $j \geq l+1$, whereas $q_l^{(l)}(\lambda_l) \neq 0$ and $\tilde{q}_l^{(l)}(\ol{\lambda_l}) \neq 0$.
Everything else follows as in Proposition \ref{prop:construction_special_basis}. 
Note that if $p \in U$ vanishes in $\lambda_1, \dots, \lambda_k$ we can use the first basis to show that $p = 0$, while in the case of $p$ vanishing in $\ol{\lambda_1}, \dots, \ol{\lambda_k}$ we can use the second basis to come to the same conclusion.  
\end{proof}

\begin{lem} \label{lem:herm_construction_with_quadr_indep}
Let $k \in \N$ and $d = \choose{k+1}{2}$. 
The binary form $f \in \R[x, y]_{2d}$ in Theorem \ref{thm:construction_polyhedral_faces_hermitian} can be constructed in such a way that for the subspace $U \subset \C[x, y]_d$ from the proof we have not only $U \ol{U} = \C[x, y]_{2d}$ but also $\dim_\C(UU) = \choose{k+2}{2}$.
\end{lem}
\begin{proof}
This is clear for $k = 1$. 
Let $k \geq 2$ and $d' = \choose{k}{2}$. 
By induction, we assume that we have a facial subspace $U' \subset \C[x, y]_{d'}$ for the Hermitian Gram spectrahedron of $g \in \R[x, y]_{2d'}$ such that $\dim_\C(U') = k$ and $U' \ol{U'} = \C[x, y]_{2d'}$ as in the proof of \ref{thm:construction_polyhedral_faces_hermitian}, 
and in addition $\dim_\C(U'U') = \choose{k+1}{2}$, i.e.~$U'$ is \emph{quadratically independent}.
Given the roots $\alpha_j$, $\ol{\alpha_j} \in \C$ of $g(x, 1)$, by Lemma \ref{lem:careful_construction_special_basis}  we find $\beta_1, \dots, \beta_k \in \C$ such that   
\[ 
\lvert \{ \alpha_j, \ol{\alpha_j} : j = 1, \dots, d' \} \cup \{ \beta_j, \ol{\beta_j} : j = 1, \dots, k \} \rvert = 2d,
\]
and that whenever $p \in U'$ vanishes in $(\beta_1 : 1), \dots, (\beta_k : 1)$ or in $(\ol{\beta_1} : 1), \dots, (\ol{\beta_k} : 1)$ then $p = 0$. 
We define $s$, $t$ and $f$ exactly as in the proof of \ref{thm:construction_polyhedral_faces_hermitian}, as well as \[
U = \C \cdot st \oplus \ol{s}U'.
\]
We have to show that the sum \[
UU = \C \cdot s^2 t^2 + s \ol{s} t U' + \ol{s}^2 U'U' 
\]
is direct. 
Then the induction hypothesis implies \[
\dim_\C(UU) = 1 + \dim_\C(U') + \dim_\C(U'U') = 1 + k + \choose{k+1}{2} = \choose{k+2}{2}.
\]
Suppose we have $s \ol{s} t u = \ol{s}^2 w$ for some $u \in U'$ and $w \in U'U'$. 
Then $stu = \ol{s}w$, so $\ol{s}$ divides $u \in U'$. 
By construction, $u = 0$. 
Hence, $s \ol{s} t U' \cap \ol{s}^2 U'U' = \{0\}$. 
Finally, $s^2 t^2 \notin s \ol{s} t U' + \ol{s}^2 U'U'$ is clear since $\ol{s}$ does not divide $s^2 t^2$.
\end{proof}

\begin{thm} \label{thm:construction_polyhedral_faces_real}
Let $k \in \N$ and $d = (k+1)^2$. 
Then there exists a positive binary form $f \in \R[x, y]_{2d}$ with distinct roots such that $\Gram(f)$ contains a simplex face $F$ 
with $(\rk(F), \dim(F)) = (2(k+1), k)$ and $\Ex(F) \subset \Ex_2(f)$.
\end{thm}
\begin{proof}
The proof heavily relies on the construction in the Hermitian case. 
Let $d_0 = \choose{k+1}{2}$ and let $g \in \R[x, y]_{2d_0}$ be a positive binary form with distinct roots such that $\H^\plus(g)$ contains a simplex face $F_0$ with the following properties (see Theorem \ref{thm:construction_polyhedral_faces_hermitian} and Lemma \ref{lem:herm_construction_with_quadr_indep}):
\begin{itemize}
\item $(\rk(F_0), \dim(F_0)) = (k+1, k)$,
\item $\Ex(F_0) \subset \Ex_1(\H^\plus(g))$,
\item the subspace $U_0 = \scrU(F_0)\subset \C[x, y]_{d_0}$ is quadratically independent, \\
 i.e.~$\dim_\C(U_0 U_0) = \choose{\dim(U_0)+1}{2} = \choose{k+2}{2}$.
\end{itemize} 
Write $\Ex(F_0) = \{p_j \otimes \ol{p_j} : j = 0, \dots, k\}$, so that $U_0 = \spn_\C(p_0, \dots, p_k)$.
Due to Proposition \ref{prop:construction_special_basis} we can choose some $q \in \C[x, y]$ of degree $d - d_0 = (k+1)^2 - \choose{k+1}{2} = \choose{k+2}{2} = \dim_\C(U_0 U_0)$ such that $f := q \ol{q} \cdot g \in \R[x, y]_{2d}$ has distinct roots and $q$ does not divide any nonzero element of $\ol{U_0 U_0}$. 
Then $q p_j \otimes \ol{q p_j}$ are rank-one tensors in the Hermitian Gram spectrahedron of $f$. 
So \[
\theta_j := \operatorname{Re}(q p_j) \otimes \operatorname{Re}(q p_j) +  \operatorname{Im}(q p_j) \otimes \operatorname{Im}(q p_j) \in \Ex_2(f)
\]
for all $j = 0, \dots, k$. 
Consider $F = \suppface( \theta_j : j = 0, \dots, k) \subset \Gram(f)$. 
For the facial subspace $U := \scrU(F) \subset \R[x, y]_d$ of $F$ we have \[
U_\C = \spn_\C(q p_j, \ol{q p_j} : j = 0, \dots, k) = q U_0 + \ol{q U_0} 
\]
Since $q$ and $\ol{q}$ are coprime and $U_0 \subset \C[x, y]_{d_0}$ with $d_0 = \choose{k+1}{2} < \choose{k+2}{2} = \deg(q)$, we see that \[
U_\C =  q U_0 \oplus \ol{q U_0} \quad \text{and} \quad \rk(F) = \dim_\R(U) = \dim_\C(U_\C) = 2 \dim_\C(U_0) = 2(k+1).
\]
Next, we have to show that the sum \[
U_\C U_\C = q^2 U_0 U_0 + \ol{q^2 U_0 U_0} + q \ol{q} \C[x, y]_{2d_0}
\]
is a direct sum. 
For $q^2 U_0 U_0 \cap \ol{q^2 U_0 U_0} = \{0\}$ we can use the same argument as above. 
Now let $q^2 u + \ol{q}^2 \ol{v} = q \ol{q} h$ for some $u, v \in U_0 U_0$ and $h \in \C[x, y]_{2d_0}$. 
Then $q(\ol{q}h - qu) = \ol{q}^2 \ol{v}$, and therefore $q$ divides $\ol{v} \in \ol{U_0 U_0}$. 
By the choice of $q$, $\ol{v} = 0$. 
Analogously, $u = 0$. 
Consequently, \[
\dim_\C(U_\C U_\C) = 2 \dim_\C(U_0 U_0) + 2d_0 + 1 = 2(d-d_0) + 2 d_0 + 1 = 2d+1. 
\]
This means that $\dim(F) = k$.
Finally, Remark \ref{rem:max_rk_min_dim_binary_forms_real} and Proposition \ref{prop:even_simplex} imply that $F$ is a simplex whose extreme points have rank two.
\end{proof}

\begin{thm}
Let $k \in \N$ and $d \geq (k+1)^2$. 
The Gram spectrahedron of a generic nonnegative binary form $f \in \R[x, y]_{2d}$ contains a simplex face $F$ with $(\rk(F), \dim(F)) = (2(k+1), k)$ and $\Ex(F) \subset \Ex_2(f)$.
\end{thm}
\begin{proof}[Sketch of proof] 
It is enough to prove the theorem for $d = (k+1)^2$.
If $Q_{2d}$ denotes the (semialgebraic) set of all $f \in \interior(\Sigma_{2d})$ with a face of the desired form in $\Gram(f)$, 
with the same argumentation as in Theorem \ref{thm:generic_herm_case}, it suffices to show that $Q_{2d}$ is dense in $\Sigma_{2d}$. 
Let $e = \choose{k+2}{2}$ and $d' = \choose{k+1}{2} = d - e$.
If $h \in \interior(\Sigma_{2d})$, we have to find $\tilde{h} \in Q_{2d}$ 'close to' $h$.
We write $h = fg$ with $f \in \interior(\Sigma_{2d'})$ and $g \in \interior(\Sigma_{2e})$. 
From the Hermitian case we know that $P_{2d'}$ is dense in $\Sigma_{2d'}$. 
Therefore, we choose $\tilde{f} \in P_{2d'}$ 'close enough to' $f$ and a suitable cofactor $\tilde{g}$ which is 'close enough to' $g$, 
and such that $\tilde{f} \tilde{g}$ is a positive binary form with distinct roots with a polyhedral face of the desired form in its (symmetric) Gram spectrahedron (cf.~the construction in Theorem \ref{thm:construction_polyhedral_faces_real}). 
Then $\tilde{f}\tilde{g} \in Q_{2d}$ is 'close to' $h$.
\end{proof}

%===================================================================%

\end{document}